\documentclass[reqno,11pt]{amsart}
\usepackage{geometry}
\usepackage{mathtools,amssymb,amsthm,mathrsfs,color,lineno,paralist,graphicx,float}
\usepackage[colorlinks,
linkcolor=blue,
anchorcolor=green,
citecolor=blue, 
]{hyperref}
\usepackage{tikz}
\usepackage{graphicx}

\usepackage{enumitem}
\usepackage[T1]{fontenc}
\usepackage[utf8]{inputenc}
\usepackage{subfig} 
\usepackage[justification = centering, labelsep =period]{caption} 

\usepackage{calc}
\definecolor{bleu1}{RGB}{0,57,128}
\def\bleu1{\color{bleu1}}

\usepackage{etoolbox}
\patchcmd{\section}{\normalfont}{\normalfont \bleu1}{}{}
\patchcmd{\subsection}{\normalfont}{\normalfont \bleu1}{}{}
\patchcmd{\subsubsection}{\normalfont}{\normalfont \bleu1}{}{}

\newtheorem{proposition}{Proposition}[section]
\newtheorem{theorem}{Theorem}[section]

\newcommand{\Z}{{\mathbb Z}}

\newcommand{\R}{{\mathbb R}}

\newcommand{\T}{{\mathbb T}}

\def\diag{{\mathrm{diag}}}

\def\0{{\mathbf 0}}

\usepackage{tikz,tikz-3dplot}
\tdplotsetmaincoords{80}{45}
\tdplotsetrotatedcoords{-90}{180}{-90}
\tikzset{surface/.style={draw=blue!70!black, fill=blue!40!white, fill opacity=.6}}

\usepackage{pgfplots}
\usepgfplotslibrary{polar}
	\usepgfplotslibrary{polar}
\pgfplotsset{compat=1.17}
\makeatletter
\tikzset{reuse path/.code={\pgfsyssoftpath@setcurrentpath{#1}}}
\makeatother

\begin{document}

	\title[]{Anderson localization of long-range quasi-periodic operators via Dynamical Rigidity } 

\author{Zhenfu Wang}
\address{
	Chern Institute of Mathematics and LPMC, Nankai University, Tianjin 300071, China
}
\email{zhenfuwang@mail.nankai.edu.cn}

 \author{Jiangong You}
 \address{
 Chern Institute of Mathematics and LPMC, Nankai University, Tianjin 300071, China} 
 \email{jyou@nankai.edu.cn}

\author{Qi Zhou}
\address{
	Chern Institute of Mathematics and LPMC, Nankai University, Tianjin 300071, China
}
\email{qizhou@nankai.edu.cn}

	\begin{abstract}
	We  establish Anderson localization for long-range quasi-periodic operators with large trigonometric potentials and Diophantine frequencies, the proof is based on   a novel dynamical rigidity argument. 
		
			\end{abstract}
	\maketitle

\section{Introduction}\label{main}Consider the long-range quasi-periodic operator $H_{v,\varepsilon w,\alpha,x}$ acting on $\ell^{2}(\mathbb{Z}^d)$:
\begin{equation}\label{long-range-op}
    (H_{v,\varepsilon w,\alpha,x}u)_n = \varepsilon \sum_{k\in\mathbb{Z}^d} w_k\,u_{n+k} + v(x+\langle n,\alpha\rangle)\,u_n, \qquad n\in\mathbb{Z}^d,
\end{equation}
where $\alpha\in\mathbb{T}^d$ is the frequency, $w_k\in\mathbb{C}$ are the Fourier coefficients of a real analytic function $w:\mathbb{T}^d \rightarrow \mathbb{R}$, and $v:\mathbb{T} \rightarrow \mathbb{R}$ is real analytic. 
Setting $v(\theta) = 2\cos(2\pi \theta)$ yields the widely studied long-range operator:
\begin{equation*}\label{longgeneral}
    (H_{2\cos,\varepsilon w,\alpha,x}u)_n = \varepsilon \sum_{k\in\mathbb{Z}^d} w_k\,u_{n+k} + 2\cos(x+\langle n,\alpha\rangle)\,u_n, \qquad n\in\mathbb{Z}^d.
\end{equation*}
The significance of this model stems from its deep connection, via Aubry duality \cite{AA}, to the one-dimensional quasi-periodic Schr\"odinger operator:
\begin{equation}\label{1.5}
    (H_{\varepsilon w, \alpha, \theta} u)_n = u_{n+1} + u_{n-1} + \varepsilon w(\theta + n\alpha) u_n, \qquad n\in\mathbb{Z}.
\end{equation}
Aubry duality remains a pivotal tool in spectral theory, instrumental in analyzing both localization-delocalization transitions (\cite{AJ2,AYZ,BJ,GJ,JK,Puig1}) and the Cantor spectrum problem (\cite{AJ3,GJY,Puig1,Puig11}).

If we further assume $w(\theta)=2\cos(2\pi \theta)$, the model reduces to the celebrated almost Mathieu operator (AMO):
\begin{equation}\label{amo}
    (H_{\lambda, \alpha,\theta}u)_n = u_{n+1} + u_{n-1} + 2\lambda\cos(\theta+ n\alpha) u_n.
\end{equation}
Peierls \cite{p} originally proposed this model to describe electrons on a 2D lattice within a uniform magnetic field \cite{harper, R}.

The study of localization in disordered media is a fundamental frontier in solid-state physics. Originating from Anderson's 1958 work \cite{anderson} on the anomalous relaxation times of electron spins in doped semiconductors, \emph{Anderson localization (AL)} is now central to understanding quantum transport, playing a pivotal role in phenomena such as the quantization of Hall conductance \cite{ag,h}. Mathematically, a Schr\"odinger operator exhibits AL if it possesses a complete basis of exponentially decaying eigenfunctions.

Despite extensive progress over the past decades, the available localization results
for quasi-periodic Schr\"odinger operators essentially fall into two distinct categories,
each with intrinsic limitations.

The first class of results concerns \emph{fixed phase} localization.
In this setting, one proves Anderson localization for a fixed phase while allowing the frequency to vary.
A landmark contribution in this direction is the work of Bourgain and Goldstein \cite{BG00}, they proved that for the one-dimensional Schr\"odinger operator $H_{v, \alpha, \theta}$ as in \eqref{1.5}, localization holds for fixed phases and typical frequencies, provided the energy lies in the positive Lyapunov exponent region. These results were later extended to the higher-dimensional setting  \cite{Bou02, Bou07,Liu, JLS}.  The proof is of non-perturbative nature, that is  the largeness of the frequeny doesn't depend on frequency.

The second class of results addresses \emph{fixed frequency} localization.  For the AMO $H_{\lambda, \alpha, \theta}$ defined in \eqref{amo}, Jitomirskaya developed a non-perturbative approach to establish Anderson localization for all Diophantine frequencies \cite{Jit94, Jit99}, and a.e. phases. This landmark result was subsequently extended to the Liouville frequency regime by Avila-You-Zhou \cite{AYZ}, and independently by Jitomirskaya and Liu \cite{JL18}.
Regarding perturbations of the AMO (more generally, for Type I operators, see \cite{GJY} for the precise definition),  corresponding results were obtained in \cite{hs, HS,GJ}.

For  general  potentials, historical results have been highly restricted to cosine or cosine-type forms. The research addressing these potentials generally follows three main methodological paths:

\begin{itemize}
\item {Multi-Scale Analysis (MSA):} Fr\"ohlich, Spencer and Wittwer used MSA to prove Anderson localization for one-dimensional Schr\"odinger operators with even $C^2$ cosine-type potentials \cite{FSW90}. This result was later extended to higher dimensions by Surace \cite{Sur96,Sur90}, and for the arithmetic case by Cao, Shi and Zhang \cite{CSZ23}. The evenness assumption was subsequently removed in one dimension by Forman and VandenBoom \cite{FV21}, and in higher dimensions by Cao, Shi and Zhang \cite{CSZ}.

\item {Perturbative argument:} Parallel to the development of MSA, Sinai established localization for one-dimensional $C^2$ cosine-type potentials using perturbative methods and geometric analysis of eigenvalues \cite{Sin87}. This line of research was further advanced by Chulaevsky and Dinaburg \cite{CD93,D}, who employed KAM-type techniques to extend these results to higher dimensions.

\item{Reducibility and Aubry Duality:} Anderson localization can also be approached by investigating the reducibility of the associated dual cocycles and leveraging Aubry duality \cite{AYZ, GY, GYZ23, GYZ24, JK}. This method translates the spectral problem into a dynamical one.
\end{itemize}

An important intermediate result is due to Eliasson \cite{Eli97,E}, who, using KAM-type
methods, proved that for a fixed Diophantine frequency and sufficiently large analytic
potentials, the associated quasi-periodic Schr\"odinger operator has pure point spectrum.
Nevertheless, the exponential decay of eigenfunctions was not established.
Very recently, Cao, Shi and Zhang  proved Anderson localization for general analytic potentials and fixed Diophantine frequencies \cite{CSZ1}.  Here, we  offer a significantly shorter  proof for trigonometric polynomials:

\begin{theorem}
Let $\alpha \in DC$\footnote{We say that $\alpha\in DC$, if there exists $\gamma>0,\tau>d-1$, such that $\inf_{j\in\Z}|\langle k,\alpha\rangle-j|\geq\frac{\gamma}{|k|^\tau},\forall k\in\Z^d\backslash \{0\}.$} and $v$ be a trigonometric polynomial. Assume that $|w_k| \leq Ce^{-c|k|}$ for some constants $C, c > 0$. Then there exists $\varepsilon_0(\alpha, v, w) > 0$, such that if $|\varepsilon|<\varepsilon_0$, $H_{v,\varepsilon w,\alpha,x}$ has AL for a.e. $x\in\T$.
\end{theorem}

We make a comment on the proof. Previous reducibility approached proof are restricted to  $\mathrm{SL}(2,\mathbb{R})$ setting,  where the fibred rotation number plays a fundamental role. By applying Aubry duality, the fibred rotation number of the reducible dual cocycle strictly determines the localized phase. However, when moving to long-range operators with general trigonometric potentials, the dual cocycle take values in  Hermitian-symplectic group $\mathrm{HSp}(2l)$\footnote{$\mathrm{HSp}(2l) = \left\{ A \in \mathbb{C}^{2l \times 2l} \,\big|\, A^* J A = J \right\}$,where $ J = \begin{pmatrix} O & -I_l \\ I_l & O \end{pmatrix} .$}. In these higher dimensions, the concept of a single fibred rotation number breaks down. One  instead obtain multiple rotational phases (from the eigenvalues of the diagonalized matrix), and these high-dimensional parameters cannot intrinsically offer specific localized phases using the traditional arguments.

To overcome the high-dimensional barrier, we reframe a structural disadvantage into a dynamical rigidity argument. Instead of attempting to extract a localized phase from a non-existent fibred rotation vectors, we start with a given phase $x$ known to admit an $\ell^2$ eigenfunction (as established by Eliasson \cite{Eli97,E}). We then prove that if the system is analytically reducible at the corresponding energy, the phases of the diagonalized cocycle must rigidly align with $x$ (Proposition \ref{thm:important}). By combining this rigidity with the measure-theoretic protection provided by the absolute continuity of the IDS \cite{WXYZ}, Aubry duality directly yields Anderson localization without resorting to heavy perturbative KAM expansions. Although this framework is currently restricted to trigonometric polynomials, we believe this dynamical approach can ultimately be extended to all analytic potentials via approximation techniques—a direction we intend to pursue in future work.

\section{Aubry duality and dynamical rigidity}

If $v(\theta)=\sum_{|k|\leq l}v_ke^{2\pi ik\theta}$ is trigonometric polynomial, we consider the dual operator $\hat{H}_{\varepsilon w, v, \alpha, \theta}$, which acts on $\ell^2(\mathbb{Z})$ as:
\begin{equation*}\label{dual-op}
(\hat{H}_{\varepsilon w, v, \alpha, \theta} u)n = \sum_{|k| \le l} v_k u_{n+k} + \varepsilon w(\theta + n\alpha) u_n, \quad n \in \mathbb{Z}.
\end{equation*} The eigenfunction $\hat{H}_{\varepsilon w, v,\alpha,\theta}u=Eu$ induced a quasiperiodic  cocycle $(\alpha,A_E(\theta))$ where 
\[
A_E(\theta)=\begin{pmatrix}
-\frac{v_{\ell-1}}{v_{\ell}} & \cdots & -\frac{v_{1}}{v_{\ell}} &
\frac{E-v_0-\varepsilon w(\theta)}{v_{\ell}} &
-\frac{v_{-1}}{v_{\ell}} & \cdots & -\frac{v_{-\ell+1}}{v_{\ell}} &
-\frac{v_{-\ell}}{v_{\ell}} \\
1 & 0 & \cdots & 0 & 0 & \cdots & 0 & 0\\
0 & 1 & \ddots & 0 & 0 & \cdots & 0 & 0\\
\vdots & & \ddots & \ddots & \vdots & & \vdots & \vdots\\
0 & \cdots & 0 & 1 & 0 & \cdots & 0 & 0\\
0 & \cdots & 0 & 0 & 1 & \ddots & 0 & 0\\
\vdots & & \vdots & \vdots & & \ddots & \ddots & \vdots\\
0 & \cdots & 0 & 0 & 0 & \cdots & 1 & 0
\end{pmatrix}.
\]
The  core of the paper is the following  Proposition \ref{thm:important}, which can be viewed as a  rigidity-type result in dynamical systems:

\begin{proposition}\label{thm:important}
Let \(h>0\), \(\varepsilon>0\), and let \(\alpha\in\mathbb{R}^d\) be irrational.
Denote by \(\mathcal R\) the set of energies \(E\) for which the dual cocycle
\((\alpha,A_E)\) is \(C_h^\omega\)-diagonally reducible, that is, there exists
$B\in C_h^\omega(\mathbb{T}^d,\mathrm{GL}(2\ell,\mathbb{C}))$
such that
\begin{equation*}\label{conjugacy}
B(\theta+\alpha)^{-1}A_E(\theta)B(\theta)
=
\Lambda
:=
\diag\bigl(e^{2\pi i\rho_1},\dots,e^{2\pi i\rho_{2\ell}}\bigr).
\end{equation*}
If \(E\in\mathcal R\) and \(x\in\mathbb{T}\) are such that \(E\) is an eigenvalue of
\(H_{v,\varepsilon w,\alpha,x}\), then we have the following:
\begin{enumerate}
    \item there exists $1\leq j\leq 2d$, $k\in\Z^d$, such that $\rho_j-x=\langle k,\alpha\rangle \mod \Z$, in particular $\Im \rho_j=0$. 
    \item the corresponding eigenfunction $u \in l^2(\mathbb{Z})$ decays exponentially: $|u_n|\le C\, e^{-2\pi h|n|}.$
\end{enumerate}
\end{proposition}
\begin{proof}
    Let \(u=(u_n)_{n\in\mathbb{Z}}\in\ell^2(\mathbb{Z}^d)\) satisfy
\(H_{v,\varepsilon w,\alpha,x}u=Eu\), and define its Fourier transform
$\widehat u(\theta)=\sum_{n\in\mathbb{Z}^d}u_n e^{2\pi i\langle n,\theta\rangle}
\in L^2(\mathbb{T}^d).$
By Aubry duality, for a.e.\ \(\theta\in\mathbb{T}^d\) the sequence
$\widetilde u_\theta(n):=\widehat u(\theta+n\alpha)\,e^{2\pi i n x},
 n\in\mathbb{Z},
$
is a nontrivial solution of the dual equation
\(\widehat H_{\varepsilon w,v,\alpha,\theta}\widetilde u_\theta=E\widetilde u_\theta\).
Define
\[
w(\theta):=\bigl(\widetilde u_\theta(2\ell-1),\dots,\widetilde u_\theta(1),
\widetilde u_\theta(0)\bigr)^{\top}
\in L^2(\mathbb{T}^d,\mathbb{C}^{2\ell}).
\]
Then $A_E(\theta)w(\theta)=e^{2\pi i x}w(\theta+\alpha)$
for a.e. $\theta$.

Set \(c(\theta):=B(\theta)^{-1}w(\theta)\).
Since \(B\) conjugates \(A_E\) to \(\Lambda\), we obtain
$\Lambda c(\theta)=e^{2\pi i x}c(\theta+\alpha).$
Writing \(c=(c_1,\dots,c_{2\ell})^\top\), this implies for each
\(1\le j\le 2\ell\) and a.e.\ \(\theta\),
\begin{equation*}\label{eq:cohom}
c_j(\theta+\alpha)=e^{2\pi i(\rho_j-x)}c_j(\theta).
\end{equation*}
Since $c(\theta)\not\equiv 0$, there exists at least one $j$ such that $c_j(\theta)\not\equiv0$. By examining the Fourier series of $c_j$, there exists a unique $k \in \mathbb{Z}^d$ such that $
\rho_j - x = \langle k, \alpha \rangle \mod \mathbb{Z},
$ which in particular implies (1) and $
c_j(\theta) = a_j e^{2\pi i \langle k, \theta \rangle},
$ where $a_j \in \mathbb{C}$ is a constant. 
The same argument applies to each component $c_j$,
every nontrivial component is a single Fourier mode, while the remaining ones
vanish identically.
Consequently, $c(\cdot) \in C_h^\omega(\mathbb{T}^d, \mathbb{C}^{2\ell})$. Since $B(\cdot) \in C_h^\omega(\mathbb{T}^d, \mathrm{GL}(2\ell, \mathbb{C}))$, it follows that $w(\cdot) \in C_h^\omega(\mathbb{T}^d, \mathbb{C}^{2\ell})$, and consequently, $\widehat{u}(\theta) = \tilde{u}_\theta(0) \in C_h^\omega(\mathbb{T}^d, \mathbb{C})$ with $
|u_n| \leq \|w\|_h e^{-2\pi h |n |}.
$
\end{proof}

\section{Proof of Theorem \ref{main}}
We assume that the potential $v(\theta)$ is a trigonometric polynomial.
There are three main ingredients for our proof. The starting point is Eliasson's famous pure point result: 
\begin{theorem}\cite{E}\label{thm:point}
Suppose that \(\alpha \in DC\), $|w_n| \leq Ce^{-c|n|}$ for some constants $C, c > 0$, $v(\cdot)\in C^\omega(\T,\R)$.
Then there exist \(\varepsilon_1=\varepsilon_1(\alpha,v,w)>0\) and a Borel measurable function  $E\colon\mathbb{T}\to\mathbb{R}$
such that for all \(|\varepsilon|<\varepsilon_1\),
 for almost every \(x\in\mathbb{T}\), the values $
E\bigl(x+\langle k,\alpha\rangle\bigr),$ $k\in\mathbb{Z}^d,$
are eigenvalues of \(H_{v,\varepsilon w,\alpha,x}\), and the corresponding eigenfunctions
\(\{u^k\}_{k\in\mathbb{Z}^d}\) belong to \(\ell^1(\mathbb{Z}^d)\) and form a complete basis of
\(\ell^2(\mathbb{Z}^d)\).
\end{theorem}
Let $\mathcal{A}\subset\mathbb{T}$ denote the full-measure set of phases
for which the conclusion of Theorem~\ref{thm:point} holds.
Fix such an $x\in \mathcal{A}$. Then the operator $H_{v,\varepsilon w,\alpha,x}$ has pure point spectrum with eigenvalues
$
E(x+\langle k,\alpha\rangle), k\in\mathbb{Z}^d.
$

Next we recall the following full measure reducibility result from \cite{WXYZ}:
    
    \begin{theorem}[{\cite[Theorem 2.3]{WXYZ}}]\label{lem:reduction}
        There exists ${\varepsilon}_2 = {\varepsilon}_2(\alpha, v, w) > 0$ such that for $|\varepsilon| \leq {\varepsilon}_2$, there exists a  zero Hausdorff dimension set $\mathcal{S} \subset \Sigma$ with the following property: For every $E \in \Sigma \setminus \mathcal{S}$, there exists $B_E \in C_{\frac{h}{4}}^\omega(\mathbb{T}^d, \mathrm{GL}(2l, \mathbb{C}))$ satisfying
        \[
        B_E^{-1}(\cdot + \alpha)A_E(\cdot)B_E(\cdot) = \diag\bigl(e^{2\pi i\rho_{1,E}},\dots,e^{2\pi i\rho_{2\ell,E}}\bigr).
        \]
 with $\rho_{i,E} \neq \rho_{j,E}$ for $i \neq j$.
\end{theorem}

As a direct consequence, we have the following: 
\begin{theorem}[{\cite[Theorem 1.2]{WXYZ}}]\label{ids}
        Let $\alpha\in DC$, $|w_k|\leq Ce^{-ck}$, $v(\cdot)$ is a trigonometric polynomial. There exists $\varepsilon_3=\varepsilon_3(\alpha,v,w)>0$, such that if $|\varepsilon|<\varepsilon_3$, IDS of $\{H_{v,\varepsilon w,\alpha,x}\}_{x\in\T}$ is  absolutely continuous.
    \end{theorem}

Fix $x\in \mathcal{A}$, define the ``bad'' eigenvalues
\[
\mathcal{E}_x
:=
\{E(x+\langle k,\alpha\rangle):\ E(x+\langle k,\alpha\rangle)\in\mathcal{S}\},
\]
where $\mathcal{S}$ is the exceptional set appearing in Lemma~\ref{lem:reduction}.
Let
 $\mathcal{E}
:=
\bigcup_{x\in\mathcal{A}}\mathcal{E}_x.$
Clearly, $\mathcal{E}\subset \mathcal{S}$. What's important for us is the following: 

\begin{proposition}\label{non}
    For almost every \(x\in\mathcal{A}\), one has \(\mathcal{E}_x = \emptyset\).
\end{proposition}

\begin{proof}
Let \(N\) be the density of states and \(\mu_{x,\delta_p}\) the spectral measure of
\(H_{v,\varepsilon w,\alpha,x}\) corresponding to the vector \(\delta_p\;(p\in\mathbb{Z}^d)\).
Since \(N\) is absolutely continuous (Theorem~\ref{ids}) and the exceptional set
\(\mathcal{S}\) has Lebesgue measure zero (Theorem \ref{lem:reduction}), we have \(N(\mathcal{S})=0\). By definition,
\[
N(\mathcal{S}) = \int_{\mathbb{T}} \mu_{x,\delta_p}(\mathcal{S})\,dx,
\qquad p\in\mathbb{Z}^d.
\]
Hence for each fixed \(p\), the integrand vanishes for almost every \(x\). Taking the
countable intersection over all \(p\), we obtain a full‑measure set
\(\mathcal{F}\subset\mathbb{T}^d\) such that
$\mu_{x,\delta_p}(\mathcal{S}) = 0, \forall\,x\in\mathcal{F}, \forall\,p\in\mathbb{Z}^d.$

Because \(\mathcal{E}_x\subset\mathcal{S}\), it follows that
\(\mu_{x,\delta_p}(\mathcal{E}_x)=0\) for all \(x\in\mathcal{F}\) and all \(p\).
Now suppose, for some \(x\in\mathcal{F}\), that \(\mathcal{E}_x\neq\emptyset\). 
By Theorem~\ref{thm:point} the operator \(H_{v,\varepsilon w,\alpha,x}\) has pure point
spectrum with eigenvalues \(\{E(x+\langle k,\alpha\rangle)\}_{k\in\mathbb{Z}^d}\).
Consequently any \(E\in\mathcal{E}_x\) is an eigenvalue and there exists \(p\in\mathbb{Z}^d\)
for which \(\mu_{x,\delta_p}(\{E\})>0\) and thus
$\mu_{x,\delta_p}(\mathcal{E}_x)\ge\mu_{x,\delta_p}(\{E\})>0,
$
contradicting \(\mu_{x,\delta_p}(\mathcal{E}_x)=0\) for $x\in \mathcal{F}$. Therefore \(\mathcal{E}_x=\emptyset\)
for all \(x\in\mathcal{F}\), i.e., for almost every \(x\).
\end{proof}

\begin{proof}[Proof of Theorem~\ref{main}]
By Proposition~\ref{non}, for almost every $x\in\mathbb{T}$ and every $k\in\mathbb{Z}^d$,
the cocycle $A_{E(x+\langle k,\alpha\rangle)}$ is \(C_{\frac{h}{4}}^\omega\)-diagonally reducible.
By Proposition~\ref{thm:important}, this implies that the eigenfunction
$u^{k}$ associated with the energy $E(x+\langle k,\alpha\rangle)$
decays exponentially.
By Theorem~\ref{thm:point}, the family $\{u^{k}\}_{k\in\mathbb{Z}^d}$
forms a complete basis of $\ell^{2}(\mathbb{Z}^d)$.
Therefore, $\{u^{k}\}_{k\in\mathbb{Z}^d}$ is a complete family of eigenfunctions of
$H_{v,\varepsilon w,\alpha,x}$, each of which decays exponentially.
\end{proof}

\section*{Acknowledgements} 
Qi Zhou would like to thank Artur Avila for insightful discussions, which ultimately led  to this note. 
 Zhenfu Wang is supported by NSFC grant (124B2011) and Nankai Zhide Foundation. Jiangong You and  Qi Zhou are supported by NSFC grant (12531006, 12526201) and Nankai Zhide Foundation.

\end{document}